\documentclass[10pt]{amsart}
\usepackage{graphics}
\usepackage[pdftex]{hyperref}
\usepackage{amssymb,amsmath,amsfonts}
\usepackage{mathtools}

\usepackage{datetime}
\usepackage{color}
\usepackage{ulem}
\usepackage{lipsum}
\usepackage{amsfonts}
\usepackage{enumerate}
\usepackage{cite}
\usepackage{tikz}
\usetikzlibrary{matrix}
\usepackage[pdftex]{hyperref}
\RequirePackage{amscd}
\RequirePackage{epic}
\RequirePackage[all]{xy}
\RequirePackage{url}
\RequirePackage[shortlabels]{enumitem}
\usepackage{empheq}
\usepackage{epsfig}
\usepackage{lineno} 
\usepackage{perpage}
\usepackage{stix}
\MakePerPage{linenumbers}
\usepackage[english]{babel}

\usepackage[utf8]{inputenc}
\usepackage{cite}
\RequirePackage{amscd}
\RequirePackage{epic}
\RequirePackage{eepic}
\RequirePackage[all]{xy}
\RequirePackage{url}
\RequirePackage[shortlabels]{enumitem}
\usepackage{empheq}
\usepackage{nomencl}
\usepackage{epsfig}
\usepackage{graphicx}

\newcommand{\eps}{\varepsilon}

\newcommand{\comment}[1]{}

\newcommand{\R}{\mathbb{R}}
\newcommand{\C}{\mathbb{C}}

\usepackage{amsthm}

\RequirePackage{url}
\newtheorem{prop}{Proposition}[section]
\newtheorem{thm}{Theorem}
\newtheorem*{thm*}{Theorem}
\newtheorem*{cor*}{Corollary}

\newtheorem{cor}{Corollary}
\newtheorem{lemma}{Lemma}
\theoremstyle{remark}
\newtheorem{ex}{Example}
\newtheorem{rmk}{Remark}[section]
\theoremstyle{definition}
\newtheorem{defn}{Definition}

\numberwithin{equation}{section}
\numberwithin{thm}{section}
\numberwithin{defn}{section}
\numberwithin{prop}{section}
\numberwithin{cor}{section}
\numberwithin{lemma}{section}
\numberwithin{rmk}{section}

\newcommand{\N}{{\mathbb N}}

\newcommand{\coD}{{\overline {\mathcal D}}}
\newcommand{\cA}{{\mathcal A}}
\newcommand{\cB}{{\mathcal B}}
\newcommand{\cC}{{\mathcal C}}
\newcommand{\cD}{{\mathcal D}}
\newcommand{\cE}{{\mathcal E}}
\newcommand{\cF}{{\mathcal F}}

\newcommand{\cK}{{\mathcal K}}
\newcommand{\cL}{{\mathcal L}}

\newcommand{\cN}{{\mathcal N}}
\newcommand{\cO}{{\mathcal O}}
\newcommand{\cP}{{\mathcal P}}
\newcommand{\cQ}{{\mathcal Q}}

\newcommand{\cS}{{\mathcal S}}

\newcommand{\cU}{{\mathcal U}}
\newcommand{\cV}{{\mathcal V}}

\newcommand{\tA}{{\mathtt{A}}}
\newcommand{\tB}{{\mathtt{B}}}
\newcommand{\tC}{{\mathtt{C}}}

\newcommand{\tE}{{\mathtt{E}}}

\newcommand{\tN}{{\mathtt{N}}}

\newcommand{\tR}{{\mathtt{R}}}

\usepackage{bm}

\renewcommand{\d}{\partial}

\newcommand{\nnorm}[1]{{\left\vert\kern-0.25ex\left\vert\kern-0.25ex\left\vert #1
	\right\vert\kern-0.25ex\right\vert\kern-0.25ex\right\vert}}

\usepackage{framed,enumitem}

\newcommand{\Pol}{\mathcal{P}(k)}
\newcommand{\dis}[2]{\mathcal{D}_{#1}(#2)}
\newcommand{\disp}[2]{\dot{\mathcal{D}}_{#1}(#2)}

\newcommand{\ucl}{\mathcal{U}(k,\rho,\cK,M)}

\newcommand{\Proj}{\mathbb{P}}

\begin{document}
\author[]{S. Barbieri}
\author[]{L. Niederman}
\address{Université Paris-Saclay and Università degli Studi Roma Tre}
\email{santiago.barbieri@universite-paris-saclay.fr }
\address{Université Paris-Saclay and IMCCE-Observatoire de Paris}
\email{laurent.niederman@universite-paris-saclay.fr }
\title[]{Bernstein-Remez inequality\\ for Nash functions: \\ a complex analytic approach.  }

\begin{abstract}
	Consider an open, bounded set $\Omega\subset \C$, a positive integer $k$ and a compact $\cK\subset \Omega$ of cardinality strictly greater than $k$. We prove that, for any function $f$ which is holomorphic in $\overline \Omega$, and whose graph satisfies $S(z,f(z))=0$ for some polynomial $S\in\C[z,w]$ of degree at most $k$ (hence $f$ is an algebraic function), the quantity $\max_{\overline\Omega}|f|/\max_{\cK}|f|$ is bounded by a constant that only depends on $k$, $\Omega$, $\cK$ but not on $f$ (estimates of this kind are called Bernstein-Remez inequalities). This result has been demonstrated by Roytwarf and Yomdin in case $\cK$ is a real interval, and later by Yomdin for a discrete set $\cK$ of sufficiently high cardinality, by using arguments of real-algebraic and analytic geometry. Here we present and extend a proof due to Nekhoroshev on the existence of a uniform Bernstein-Remez inequality for algebraic functions, which relies on classical theorems of complex analysis. Nekhoroshev's work remained unstudied despite its important consequences in Hamiltonian dynamics and is here presented and extended in a self-contained and pedagogical way, while the original reasonings were rather sketchy. 
\end{abstract}
\maketitle
\section{Introduction}
\subsection{Bernstein-Remez inequality and its application in semi-algebraic geometry}
\

In the literature, one kind of Bernstein's inequality relates the upper bound of a function $f$ - which is analytic in an open bounded domain $\Omega\subset {\mathbb C}$ and continuous in $\overline\Omega$ - with the maximum attained by the same function on a compact subset $\cK\subset \Omega$. In the sequel, this kind of estimate will be called Bernstein-Remez inequality in order to avoid confusion with other sort of Bernstein's inequalities that involve derivatives or primitives (see e.g. \cite{Fefferman_Narasimhan_1995}). The first example of this estimate was given by S. Bernstein in 1913 for polynomial functions:
\begin{thm}{(Bernstein, \cite{Bernstein_1912}, p.14-15)}
	
	Let $P(x)$ be polynomial of one variable with real coefficients and of degree $k\in\N$. Consider the complex domain $\cE$ bounded by the ellipse with foci in $(-1,0), (1,0)$ and whose sum of the semi-axes is equal to $R>0$. Then one has
	$$
	\max_{z\in\cE}|P(z)|\le R^k \max_{x\in [-1,1]}|P(x)|\ .
	$$
\end{thm}
By adopting the notations of \cite{Roytwarf_Yomdin_1998}, we give the following
\begin{defn}[Bernstein's constant and uniform inequality]\label{bernie}
	Let $\Omega\subset\C$ be an open bounded domain, $\cK\subset \Omega$ be a compact and let $f:\Omega\longrightarrow \C$ be holomorphic in $\Omega$ and continuous in $\overline{\Omega}$. The {\it Bernstein's constant of $f$} with respect to $\Omega,\cK$ is the quantity
	$$
	\tB(f,\cK,\Omega):=\max_{\overline\Omega}|f|/\max_{\cK}|f|\ .
	$$
	Any family $\cF$ of holomorphic functions defined in $\Omega$ and continuous in $\overline\Omega$ is said to satisfy a {\it uniform Bernstein-Remez inequality} if there exists $\tC(\cK,\Omega) >0$ such that for all $f\in\cF$
	$$
	\max_{\overline\Omega}|f|\le \tC(\cK,\Omega)\,\max_{\cK}|f|\qquad \text{ or, equivalently, if } \qquad \sup_{f\in\cF}\,\tB(f,\cK,\Omega)\le \tC(\cK,\Omega)\ .
	$$
	
\end{defn}

Finding classes of functions admitting a uniform bound on their Bernstein's constants - and thus satisfying a uniform Bernstein-Remez inequality - is an important issue in the study of several problems of analysis, algebraic geometry and dynamics. In this paper we will establish the existence of a uniform Bernstein-Remez inequality for the following class of analytic-algebraic (Nash) functions: 

\begin{defn}\label{AlgFunct}
	Consider $k\in\mathbb{N}$, $\rho >0$ and denote with $\cD_\rho(0)$ the open complex disk of radius $\rho$ centered at the origin. 

 We indicate with ${\mathcal{V}} (k,\rho )$ is the set of functions that satisfy:
	
	\begin{enumerate}

	\item  $f$ is holomorphic over $\cD_\rho(0)\,$;
	
	\item The graph of $f$ is included in an algebraic curve 
$$\tR_S :=\{(z,w)\in\C^2:S(z,w)=0\}$$ 
associated to  a non-zero polynomial $S\in{\mathbb{C}} [z,w]$ of degree at most $k$, hence 
$$
S(z,f(z))=0\qquad \text{ for } z\in\cD_\rho(0)\ ;
$$
	
	\item The algebraic curve $\tR_S$ is such that $\tR_S\cap \{\cD_\rho(0)\times\C\}$ is the union of at most $k$ elements that can be either vertical lines of the form $\{ (z,w)\in\C^2\ |\ z=z_*\}$ or disjoint graphs of holomorphic functions over $\cD_\rho(0)$.
	
\end{enumerate}

\end{defn}

\medskip 

 The main result developed here is the following:

\begin{thm}\label{Bernstein0}
		
  With the notations of Definition \ref{AlgFunct}, consider a compact set $\cK\subset \cD_\rho(0)$ satisfying:
	\begin{equation}
		0\in\cK\ \text{\it and }\ \text{card }(\cK)> k.
	\end{equation}
	
	The functions of the family ${\mathcal{V}} (k,\rho)$ satisfy an uniform Bernstein-Remez inequality with respect to $\cK$ and to  any open set $\Omega$ such that $\cK\subset\Omega$ and $\overline{\Omega}\subset\cD_\rho(0)$. 
	
	Consequently, there exists a number $\tC=\tC(k,\rho,\cK,\Omega )>0$ such that, for any $f\in{\mathcal{V}} (k,\rho)$, one has:
	$$
	\max_{z\in {\overline \Omega}}|f(z)|\le\tC\, \max_{z\in \cK}|f(z)|\ .
	$$
	
\end{thm}

\begin{rmk} The condition $0\in\cK$ is not mandatory but only for convenience.
\end{rmk}

This theorem has been demonstrated by Roytwarf and Yomdin in \cite{Roytwarf_Yomdin_1998} in the cases where $\cK =[-\rho',\rho']\subset\R$ or $\cK ={\overline \cD}_{\rho'}(0)\subset\C$, and $\Omega=\cD_{\rho''}(0)\subset\C$, with $0<\rho'<\rho''<\rho$. Moreover, the authors obtain quantitative estimates on the upper bound $\tC(k,\rho',\rho'',\cK )$ for the Bernstein's constant and they generalize these results to relevant cases of algebraic families of holomorphic functions. More recently, these estimates have been extended by Yomdin and Friedland to the case of a discrete compact $\cK$ of sufficiently high cardinality in refs. \cite{Yomdin_2011} and \cite{Friedland_Yomdin_2017}, thanks to the introduction of a geometric invariant related to entropy.

\medskip

 Important applications of the Bernstein-Remez inequalities will be presented in the sequel, but first we would like to focus on semi-algebraic geometry. In refs. \cite{Yomdin_2008} and \cite{Yomdin_2015}, Yomdin has shown that - with the exception of a small part - any two-dimensional semi-algebraic set can be covered by the images of a finite number of  real-analytic, algebraic charts of the interval $[-1,1]$. Moreover, thanks to the existence of a Bernstein-Remez inequality for algebraic functions, one has a control on the size of all the derivatives of these charts that depends only on the order of the derivation and on the degrees of the polynomials involved in the definition of the considered semi-algebraic set. This is a partial extension of the theorem (called algebraic lemma) about the ${\mathcal{C}}^k -$reparametrization of semi-algebraic sets due to Yomdin  \cite{Yomdin_1987} and Gromov \cite{Gromov_1987}. The analytic reparametrization result has recently been generalized (see \cite{Binyamini_Novikov_2019} and \cite{Cluckers_Pila_Wilkie_2020}) to higher dimensional sets with more general structures than semi-algebraic, which allows for important applications in arithmetics. 

\subsection{Application to Hamiltonian dynamics and Nekhoroshev theory}
The authors discovered the Bernstein-Remez inequality during the investigation of an important result of Hamiltonian dynamics.
Namely, as it is known, hamiltonian systems that are integrable in the sense of Arnol'd-Liouville (see e.g. \cite{Arnold_1989}) constitute a very relevant, though exceptional, class. Moreover, many important physical systems (especially in Celestial Mechanics) can be modeled by a hamiltonian system which is a small perturbation of an integrable one \cite{Arnold_1989}. During the 1970s, Nekhoroshev\footnote{See \cite{Nekhoroshev_1977}, or \cite{Guzzo_Chierchia_Benettin_2016} for a more modern presentation} proved that if we consider a real-analytic, integrable hamiltonian whose gradient satisfies a suitable, quantitative transversality condition known as {\it steepness} then, for any sufficiently small perturbation of this initial integrable hamiltonian, the solutions of the perturbed system are stable and have a very long time of existence\footnote{The time of stability is exponential (polynomial) in the inverse of the size of the perturbation if the total hamiltonian belongs to the Gevrey (Hölder) class. See \cite{Marco_Sauzin_2002}, \cite{Bounemoura_2010}, \cite{Barbieri_Marco_Massetti_2022}.}. The original definition of steepness given by Nekhoroshev is quite involved. In order to grasp an idea of what this property means, it is worth to mention that a real-analytic function is steep if and only if it has no isolated critical points and if any of its restrictions to any affine proper subspace has only isolated critical points (see \cite{Ilyashenko_1986} and \cite{Niederman_2006}). 

Nekhoroshev also proved in \cite{Nekhoroshev_1973} that the steepness condition is generic, both in measure and topological sense: the Taylor polynomials of sufficiently high order of non steep functions are contained in a  semi-algebraic set having positive codimension in the space of polynomials. Hence, steep functions are characterised by the fact that their Taylor polynomials satisfy suitable algebraic conditions (see \cite{Nekhoroshev_1979} and \cite{Barbieri_2020}). Although these results have been studied and extended for more than fourty years (so that {\it Nekhoroshev Theory} is a classic subject of study in the community of dynamical systems), the proof of the genericity of steepness has remained, up to now, largely unstudied and poorly understood. This is certainly due to the fact that such a demonstration does not involve any arguments of dynamical systems, but combines quantitative reasonings of real-algebraic geometry and complex analysis. It is precisely in those reasonings that the Bernstein-Remez inequality plays a major rôle.

Namely, a crucial step in Nekhoroshev's proof of the genericity of steepness consists in considering, for any fixed polynomial $P\in {\mathbb{C}}[X_1,\ldots ,X_n]$, the semi-algebraic set - called thalweg nowadays (see \cite{Bolte_Daniilidis_Ley_Mazet_2010}) - defined by 
$$
{\mathcal{T}}_P\subset{\mathbb{R}}^n:=\{u\in\R^n|\ \vert\vert\nabla P(u)\vert\vert\leq\vert\vert\nabla P(v)\vert\vert\ \forall v\in\R^n \text{  s.t. } \vert\vert u\vert\vert =\vert\vert v\vert\vert\}\ .
$$

Nekhoroshev shows that, for any open ball ${\mathcal{B}}\subset{\mathbb{R}}^n$, the intersection ${\mathcal{T}}_P\cap{\mathcal{B}}$ contains a real analytic curve ${\mathcal C}$ such that both the distance between the extremities of $\mathcal C$ and the complex analyticity width of its parametrization admit a lower bound that only depends on the degree of the polynomial $P$. More specifically, ${\mathcal C}$ can be parametrized by Nash functions and the existence of a Bernstein-Remez inequality (also proved in \cite{Nekhoroshev_1973}) ensures uniform upper bounds on the derivatives of these charts.  Actually, this result about the thalweg in \cite{Nekhoroshev_1973} is a particular case of a general theorem due to Yomdin \cite{Yomdin_2008} mentioned previously about analytic reparametrizations of semi-algebraic sets. The uniform control on the parametrization of the curve $\cC$ is unavoidable in \cite{Nekhoroshev_1973}, since it ensures that - for a smooth function - steepness is an open property which can be determined by the Taylor expansion at a certain order (we have a "finite-jet" determinacy of steepness). In that way, the study of the genericity of steepness is reduced to a finite-dimensional setting which involves polynomials of bounded order and this is crucial in order to prove the genericity. The study of \cite{Nekhoroshev_1973} will be investigated and specified in a forthcoming paper of the first author. 

 From a more general point of view, the steepness condition is introduced to prevent the abundance of rational vectors on certain sets. In particular, deep applications of the controlled analytic parametrizations of semi-algebraic sets - yielding bounds on the number of integer points in semi-algebraic sets - are given in \cite{Binyamini_Novikov_2019} and \cite{Cluckers_Pila_Wilkie_2020}. Along these lines of ideas, Yomdin-Gromov algebraic lemma with tame parametrizations of semi-algebraic sets (see \cite{Yomdin_1987}, \cite{Gromov_1987}) was used by Bourgain, Goldstein, and Schlag \cite{Bourgain_Goldstein_Schlag_2002} to bound the number of integer points in a two dimensional semi-algebraic set.

In ref. \cite{Nekhoroshev_1973}, Nekhoroshev proves the existence of a Bernstein-Remez inequality for algebraic functions in his specific problem, by exploiting arguments of real algebraic geometry and by making an intensive use of complex analysis. The original statements are difficult to disentangle from the context of the genericity of steepness and the proofs are very sketchy. Roytwarf-Yomdin \cite{Roytwarf_Yomdin_1998} and Yomdin \cite{Yomdin_2011} combine reasonings of real algebraic and analytic geometry together with an important result due to Biernacki \cite{Biernacki_1936} on the Taylor's coefficients of $p$-valent functions\footnote{An analytic function over a disc is said to be $p$-valent if either it is constant or 
each element of ${\rm Im}(f)$ is the image of at most $k$ points. Any algebraic function $f$ satisfying $S(z,f(z))=0$ for some polynomial $S\in\C[z,w]$ of degree $k$ is $k$-valent (Lemma \ref{valency}).}. Nekhoroshev's different strategy of proof is briefly mentioned in \cite{Roytwarf_Yomdin_1998} (p. 848) without quoting \cite{Nekhoroshev_1973}, but so far we have not been able to find any reference that shows it in detail except for the original paper (see \cite{Nekhoroshev_1973}, Lemma 5.1, p.446).

This is our motivation for a short, pedagogical self-contained exposition of Nekhoroshev's proof which relies on standard theorems of complex analysis. Actually, Nekhoroshev \cite{Nekhoroshev_1973} shows the existence of a Bernstein-Remez inequality only in the case in which the compact $\cK$ is a real segment and the considered algebraic functions have a particular form, since this is sufficient for his purposes. Here, we extend this strategy by considering any compact set $\cK$ of high enough cardinality and we get rid of the additional conditions on the form of the algebraic functions.

Nekhoroshev's approach presents two drawbacks. It does not permit to obtain quantitative estimates for the Bernstein constants as in \cite{Roytwarf_Yomdin_1998} and \cite{Yomdin_2011}. Moreover, we were not able to prove a Bernstein-Remez inequality for an algebraic function on its maximal disk of regularity, what is obtained in \cite{Roytwarf_Yomdin_1998} and  is called structural inequality, but only for the maximal disk of regularity of all the algebraic functions associated to the considered polynomial.  Nevertheless, these two points are not relevant for certain applications of the Bernstein-Remez inequality, especially in Nekhoroshev's arguments on the thalweg.

Finally, as it was already known in \cite{Nekhoroshev_1973} and is central in \cite{Roytwarf_Yomdin_1998}, the existence of uniform Bernstein's constants implies uniform bounds on the Taylor coefficients of algebraic functions. In this spirit, we shall also state a result of this kind in Theorem \ref{Cauchy}.

\subsection{Other applications of the Bernstein-Remez inequality}
\

In complex analysis, Bernstein's constants appear when estimating the number of zeros of holomorphic functions. Namely one has the following result, which is a consequence of Jensen's formula:
\begin{thm}{(Tijdeman, \cite{Tijdeman_1971})}
	
	Take three real numbers $R,t>0$, $s>1$ and let $f$ be a holomorphic function of the closed disk $\coD_{(st+s+t)R}(0):=\{z\in\C:|z|\le (st+s+t)R\}$, with $f\not\equiv 0$. Denote with $N(0,R,f)$ the number of zeros of $f$ in $\coD_{R}(0)$. Then
	$$
	N(0,R,f)\le \frac{1}{\log\, s}\log\, \tB(f, \coD_{tR}(0), \coD_{(st+s+t)R}(0))\ .
	$$
\end{thm}
Hence, there exists a uniform bound on the number of zeros for the elements of a family of holomorphic functions admitting a uniform Bernstein-Remez inequality.  It is known \footnote{See e.g. \cite{Roytwarf_Yomdin_1998}, \cite{Francoise_Yomdin_1997} and references therein for a deeper discussion on the subject.} that this fact proves fundamental in investigations around the second part of Hilbert's 16th problem, whose aim is to study the relative positions and to bound the number of limit cycles for any dynamical system defined by a real polynomial vector field of fixed degree $k$ on the plane. 

In a rather different context, uniform Bernstein's constants with respect to the closed complex and real ball have been obtained in \cite{Fefferman_Narasimhan_1995} for functions belonging to a finite span of holomorphic functions of $n$ variables depending real-analytically on a parameter $\lambda\in\cL\subset\R$, with $\cL$ a compact. This result is closely related to investigations in PDEs in which algebraic functions appear when studying the symbol of a pseudo-differential operator (see \cite{Fefferman_Narasimhan_1994} and references therein).

The paper is organized as follows: section 2 contains the mathematical setting and the results, whereas section 3 contains their proofs. Section 4 is devoted to the proof of some technical Lemmas that are used in section 3 and is the "core" of Nekhoroshev's strategy (especially Lemma \ref{panettone}). Finally, we have relegated to the appendices the statements of some classic results on algebraic curves, complex analysis and real-algebraic geometry that are used throughout the paper.
\section{Setting and main results}
\subsection{Setting} For any $r>0$ and any $z_0\in\C$, we denote with $\cD_r(z_0)$ the open complex disk centered at $z_0$ and with $\coD_r(z_0)$ its closure.

$\C[z,w]$ indicates the ring of polynomials of two variables over the complex field. Throughout this paper, we will often identify $\C[z,w]$ with $\C[z][w]$, the ring of complex polynomials in $w$ over the ring of polynomials of the complex variable $z$.

 For $ k\in\N$, we indicate with $\cQ(k)\subset \C[w]$ and $\Pol\subset \C[z,w]$ respectively the subspaces of complex polynomials in one and two variables having degree inferior or equal to $k$. Since $\cQ(k),\Pol$ are finite-dimensional, they can be equipped with an arbitrary norm.

\subsection{Main results}

With the notations of Theorem \ref{Bernstein0}, we consider the following class of functions:

\begin{defn}\label{vu}
	For $k\in\N$ and $\rho >0$, we denote by $\cV_0 (k,\rho )\subset {\mathcal{V}} (k,\rho)$ the subset of those functions $g\in {\mathcal{V}} (k,\rho)$ that satisfy $g(0)=0$.
\end{defn}

The functions of the family $\mathcal{V}_0 (k,\rho )$ belong to the same Bernstein's class w.r.t. the sets $\Omega$ and $\cK$ of Theorem \ref{Bernstein0}. Namely, one has:

\begin{thm}\label{Bernstein}
	Consider an open set $\Omega$ satisfying $\overline{\Omega}\subset \cD_\rho(0)$ and $\cK\subset \Omega$ a compact satisfying $\text{ card }\cK>k$.  There exists a number $\tC_0 =\tC(k,\rho,\cK,\Omega)>0$ that bounds uniformly the Bernstein's constants of the elements of $\cV_0  (k,\rho )$, i.e.: 
	$${\rm for\ any}\  g\in\cV_0  (k,\rho ),\ {\rm one\ has}\ 
	\max_{z\in \overline \Omega}|g(z)|\le\tC_0\, \max_{z\in \cK}|g(z)| .
	$$
\end{thm}

\begin{rmk} The hypothesis $0\in\cK$ of Theorem \ref{Bernstein0} is unnecessary in Theorem \ref{Bernstein}.
\end{rmk}

 Theorem \ref{Bernstein0} is a consequence of Theorem \ref{Bernstein} since, for any $f\in\cV(k,\rho)$, the function $g(z):=f(z)-f(0)$ belongs to the class $\cV_0(k,\rho)$ and Theorem \ref{Bernstein} ensures:
\begin{align*}
	\begin{split}
		\max_{\overline \Omega}|f|\le & |f(0)|+\max_{{\overline \Omega}}|g|\le |f(0)|+\tC_0\,\max_{\cK}|g|\\
		\le & |f(0)|+\tC_0|f(0)|+\tC_0\max_{\cK}|f|= (1+2\tC_0)\max_\cK|f| :=\tC\,\max_{ \cK}|f|
	\end{split}
\end{align*}
where the last estimate comes from the hypothesis $0\in\cK$. 

 This concludes the proof of Theorem \ref{Bernstein0}

\medskip

Theorem \ref{Bernstein} is also the cornerstone which allows to prove an uniform upper bound on the Taylor coefficients of functions in $\cV_0(k,\rho )$. More specifically, we introduce the following class of bounded algebraic functions:

\begin{defn}\label{Def2.3}
	With the previous notations, for any $M\geq 0$ and any compact $\cK\subset\cD_{\rho}(0)$, we denote with $\cU(k,\rho,\cK, M)$ the subset of those functions $g\in\cV_0(k,\rho )$ that satisfy $\max_{\cK}|g|= M$.

 Hence, we have $\cV_0(k,\rho )=\cup_{M\geq 0}\cU(k,\rho,\cK, M)$.
\end{defn}

The functions in $ \ucl$ fulfill a generalized uniform Cauchy inequality, namely
\begin{thm}\label{Cauchy}

 Under the additional assumption $\text{ card }\cK>k$, there exists a constant $K=K(k,\rho,\cK)$ such that, for any function $g\in \ucl$, the coefficients of the Taylor series
	\begin{equation}\label{Taylor}
		g(z)=\sum_{j=1}^{+\infty}a_jz^j\ \ ({\it with}\ g(0)=0)
	\end{equation}
	satisfy the uniform inequality
	\begin{enumerate}
		\item $$|a_j|\le K(k,\rho,\cK) M\qquad \text{  if } \rho>1\ ;$$
		\item for any number $m>1$
		$$
		|a_j|\le K(k,m,\cK) M\left(\frac{m}{\rho}\right)^j\qquad \text{  if } \rho\le 1\ .
		$$
	\end{enumerate}
	
\end{thm}

\begin{rmk}	This result is stated and used in \cite{Nekhoroshev_1973} in the particular case where $\rho>1$, $\cK=[0,\lambda]\subset\R$, $M(\lambda)=\lambda$ and $\lambda>0$. The equivalence between a uniform bound on the growth of the Taylor coefficients and the Bernstein-Remez inequality is central in \cite{Roytwarf_Yomdin_1998}.
\end{rmk}

Theorems \ref{Bernstein} and \ref{Cauchy} will be proved in the next section.
\section{Proof of the main results}
 We first need the following standard lemma:
\begin{lemma}\label{valency}
	With the notations of the previous section, an analytic-algebraic (Nash) function $f$, associated to a polynomial $S\in \mathbb{C}[z,w]$ of degree $k\in \N$, is $k$-valent: that is, if $f$ is not constant then each element of ${\rm Im}(f)$ is the image of at most $k$ points.
Consequently, if $f$ is not identically zero, then $f$ cannot be identically zero over any set $\cK$ included in the domain of definition of $f$ such that ${\rm Card}(\cK)>k$.
\end{lemma}
\begin{proof}  Assume, by absurd, that $f$ is non-constant and that there exists $w_0\in{\rm Im}(f)$ which is the image of at least $p>k$ points.  The polynomial $S^{w_0} (z):=S(z,w_0)$ would admit $p>k$ roots while ${\deg}(S^{w_0})\leq k$ by hypothesis. The Fundamental Theorem of Algebra ensures that $S^{w_0}$ must be identically zero and one has the factorization $S(z,w)=(w-w_0)^\alpha{\hat S}(z,w)$, 
	where $\alpha\in\{1,...,k\}$, while $\hat S$ cannot be divided by $(w-w_0)$ in $\C[z,w]$. Since $f$ is analytic and not constant, then the preimage $f^{-1}(\{ w_0\})$ is a discrete set and the graph of $f$ must fulfill ${\hat S}(z,f(z))=0$ out of $f^{-1}(\{ w_0\})$. By continuity, one has $\hat S(z,f(z))=0$ on the whole domain of definition of $f$ since $f^{-1}(\{ w_0\})$ is discrete. But $\deg \hat S^{w_0}\le k$, with $\hat S^{w_0}(z):=\hat S(z,w_0)$, and $\hat S^{w_0}$ admits more than $k$ roots, hence the previous argument ensures that $\hat S$ can be divided by $(w-w_0)$, in contradiction with the construction.

 Moreover, if $f\not\equiv 0$, then $0$ admits at most $k$ inverse images by $f$, and $f$ cannot be identically null over any set $\cK$ included in the domain of definition of $f$ satisfying $\text{card }\cK>k$.
\end{proof}

 Consequently, without any loss of generality, in Theorem \ref{Bernstein} we can assume $g\in\cU(k,\rho,\cK,1)$ according to Definition \ref{Def2.3} (hence $g\in\cV_0(k,\rho )$ and $\max_\cK|g|= 1$) since, if this is not the case, it suffices to consider $g/\max_\cK|g|$.

Then, we define the following set:
\begin{defn}\label{branch}
	$\mathcal A:=\mathcal A(\mathcal K,k,\rho)$ denotes the set of those polynomials $S\in \Pol\backslash\{0\}$ whose algebraic curve $\tR_S :=\{(z,w)\in\C^2:S(z,w)=0\}$ satisfies
	\begin{enumerate}
	\item 	$\tR_S\cap\{ \cD_\rho(0)\times\C\}$ is the union of at most $k$ elements that can be either vertical lines of the form $\{ (z,w)\in\C^2\ |\ z=z_*\}$ or disjoint graphs of holomorphic functions over $\cD_\rho(0)$\,;
	\item 	there exists $g_S\in\mathcal U(k,\rho,\mathcal K,1)$ whose graph is contained in $\tR_S\cap\{ \cD_\rho(0)\times\C\}$. 
	\end{enumerate}
\begin{rmk}
	For any $S\in\cA$, the function $g_S$ is unique since the graphs contained in the algebraic curve of $S$ are disjoint over $\cD_\rho(0)$ and the value $g_S(0)=0$ is fixed. 
\end{rmk}	 
	
\end{defn}

The central property in the proof of Theorem \ref{Bernstein} is the following
\begin{lemma}\label{chiusura}
	$\cA\cup\{0\}$ is closed in $\Pol$ and, for any open set $\Omega$ satisfying $ \cK\subset\Omega$, $\overline\Omega\subset  \cD_\rho(0)$, the function
	$$
	\mu_{\Omega}:\cA\longrightarrow \R \qquad S\longmapsto \max_{\overline \Omega} |g_S|
	$$
	is continuous.
\end{lemma}
We shall relegate the proof of Lemma \ref{chiusura} to the next section and we shall exploit its statement here to prove Theorems \ref{Bernstein} and \ref{Cauchy}.
\begin{proof}{\it (Theorem \ref{Bernstein})}

 By Definitions \ref{vu}, \ref{Def2.3} and \ref{branch}, we can associate to any  $g\in\cU(k,\rho,\cK,1)$ a polynomial $S\in\cA$ such that $g=g_S$. A standard combinatorial computation yields that $\Pol$ is isomorphic to $\C^m$, with $m=(k+1)(k+2)/2$.
	It is also easy to see that for any polynomial $S\in\cA$ and for any $c\in\C\backslash\{0\}$ the polynomial $S'=cS$ belongs to $\cA$ and $g_{S'}\equiv g_S$, so that it makes sense to pass to the projective space
	$$
	\C\Proj^{m-1}:=\{\C^m\backslash\{0\} \}/ \{\C\backslash \{0\}\}\quad ,
	\qquad
	\pi: \C^m\backslash\{0\}\longrightarrow \C\Proj^{m-1}\ ,
	$$
	where $\pi$ denotes the standard canonical projection inducing the quotient topology in $\C\Proj^{m-1}$. Moreover, for any open set $\Omega$ satisfying $\cK\subset \Omega $, $ \overline \Omega \subset  \cD_\rho(0)$, the function
	$$
	\hat\mu_{\Omega}:\pi(\cA)\longrightarrow \R\quad ,\qquad \pi(S)\longmapsto \max_{\overline \Omega} |g_S|
	$$
	is well defined and continuous by Lemma \ref{chiusura}. To prove the latter claim, take a closed set $\cE\subset\R$ and consider its inverse image $\hat \mu^{-1}_{\Omega}(\cE)=\pi(\mu^{-1}_{\Omega}(\cE))$. Since $\mu_{\Omega}$ is continuous, $\mu^{-1}_{\Omega}(\cE)$ is closed in $\cA$ for the induced topology. By Lemma \ref{chiusura}, $\cA\cup\{0\}$ is closed in $\C^m$, so that $\cA$ is closed in $\C^m\backslash\{0\}$. Hence, $\mu^{-1}_{\Omega}(\cE)$ is closed in $\C^{m}\backslash\{0\}$. Since $\mu_\Omega$ is invariant if its argument is multiplied by a complex non-zero constant, $\mu^{-1}_\Omega(\cE)$ is saturated and one has $\pi^{-1}(\pi(\mu^{-1}_\Omega(\cE)))=\mu^{-1}_\Omega(\cE)$. Consequently, the set  $\pi(\mu^{-1}_{\Omega}(\cE))=\hat \mu^{-1}_{\Omega}(\cE)$ is closed for the quotient topology because its inverse image w.r.t. $\pi$ is closed. This proves the continuity of $\hat\mu_\Omega$.
	
	Moreover, since $\cA$ is closed and saturated in $\C^m\backslash\{0\}$, $\pi(\cA)$ is closed in $\C\Proj^{m-1}$ and the compacity of $\C\Proj^{m-1}$ ensures that $\pi(\cA)$ is compact. By continuity of $\hat\mu_{\Omega}$, the image $\hat\mu_{\Omega}(\pi(\cA))$ is a compact set of $\R$, hence bounded. Therefore, there exists a constant $\tC(k,\rho,\cK,\Omega)$ such that for any $g\in \cU(k,\rho,\cK,1)$ one has $$
	\max_{\overline \Omega}|g|=\displaystyle \frac{\max_{\overline \Omega}|g|}{\max_{\cK}|g|}\le \tC(k,\rho,\cK,\Omega)
	$$
	and this concludes the proof.
\end{proof}

\begin{proof}{\it (Theorem \ref{Cauchy})}
 
 Since $g$ is non identically zero over $\cK$ (see  Lemma \ref{valency}), we can consider the function $g/M$ and we are reduced to the case $M=1$. 
	
 We start by proving the statement for $\rho>1$. 

	By absurd, suppose that Theorem \ref{Cauchy} is false. By choosing the analyticity width $\varrho=1<\rho$, the standard Cauchy estimates at the origin yield $|a_j|\le\max_{|z|=1}|g(z)|$ for any $j\in\N$ and, since the thesis is false, there exists a sequence $\{g_n\}_{n\in\N}$, with $g_n\in \cU(k,\rho,\cK,1)$ for all $n\in\N$, satisfying
	$$
	\lim_{n\longrightarrow+\infty}\left(\max_{|z|=1}|g_n(z)|\right)=+\infty \ .
	$$
	
	By the maximum principle over $\overline{\mathcal D}_1(0)$, all the functions in the sequence
	$$
	\bar{g}_n(z):=\frac{g_n(z)}{\max_{|z|=1}|g_n(z)|}\qquad n\in\N
	$$
	belong to $\mathcal U(k,\rho,\overline{\mathcal D}_1(0), 1)\subset \cV_0(k,\rho )$. Hence, by Theorem \ref{Bernstein}, for any open set $\Omega$ satisfying $\overline{\mathcal D}_1(0)\subset  \Omega$, $\overline\Omega\subset \cD_\rho(0)$ and for any $n\in \N$, the Bernstein-Remez inequality
	$$
	\max_{z\in \overline\Omega}|\bar{g}_n(z)|\le\tC_1\, \max_{|z|\leq 1}|\bar{g}_n(z)|=\tC_1\, \max_{|z|=1}|\bar{g}_n(z)|=\tC_1
	$$
	holds for some constant $\tC_1=\tC_1(k,\rho,\overline{\mathcal D}_1(0),\Omega)$.

 Since an arbitrary compact subset of $\cD_\rho(0)$ can be included in a bounded open set $\Omega\in \cD_{\rho}(0)$ such that $\coD_1(0)\subset \Omega$ and $\overline\Omega\subset \cD_{\rho}(0)$, we see that $\{\bar{g}_n\}_{n\in\N}$ is a locally bounded sequence. Hence, by Montel's Theorem (Corollary \ref{montel}), one can extract a subsequence $\{\bar{g}_{n_j}\}_{j\in\N}$ which converges locally uniformly in $\cD_\rho(0)$ to a holomorphic function $\bar{g}$. 
	
	 With Lemmas \ref{sequence} and \ref{valency}, $\bar g$ is algebraic and $k$-valent in $\cD_{\rho}(0)$.
 Now, for any fixed $z^*\in \cK$, one has the pointwise limit
	$$
	\bar g(z^*)\! =\! \! \lim_{j\longrightarrow+\infty}\bar{g}_{n_j}(z^*)\! \le\! \!  \lim_{j\longrightarrow+\infty} \frac{\max_{z\in\cK}|g_{n_j}(z)|}{\max_{|z|=1}|g_{n_j}(z)|}\! =\! \! \lim_{j\longrightarrow+\infty} \frac{1}{\max_{|z|=1}|g_{n_j}(z)|}=0.
	$$
 On the one hand, since $\text{ card }\cK>k$, the previous expression implies that $\bar g$ has at least $k+1$ zeros in $\cD_\rho(0)$. On the other hand, we have shown that $\bar g$ is $k$-valent, so the number of its zeros must be bounded by $k$. We have derived a contradiction and therefore the statement is true for $\rho>1$. 
	
\medskip

 In case $\rho\le 1$, for any fixed $m>1$ one considers the function
	$$ g_{m}(z):=g\left(\frac{\rho}{m}z\right):=\sum_{j=1}^{+\infty}c_jz^j=\sum_{j=1}^{+\infty}a_j\left(\frac{\rho}{m} z\right)^j
	$$
	analytic in $ \cD_m(0)$ and belonging to $\cU(k,m,\cK_m,1)$, where 
	$$
	\cK_m:=\{z\in\cD_m(0):\frac{\rho}{m}z\in \cK\}
	$$ 
	fulfills $\text{ card}\, \cK_m>k$ since $\cK$ does.
	
	Since the convergence radius of $g_{m}$ is $m>1$, the statement holds for this function and there exists a constant $K(k,m,\cK)$ such that
	$$
	|c_j|\le K(k,m,\cK) \qquad \forall \ j\in \N,
	$$
	which implies
	$$
	|a_j|\le K(k,m,\cK) \left(\frac{m}{\rho}\right)^j\ .
	$$
	This concludes the proof.
\end{proof}

\section{Technical lemmas}\label{technical_lemmas}

The aim of this section is to prove Lemma \ref{chiusura}. We first recall a few classical points.

The algebraic curve of a polynomial $S\in\C[z,w]$ is the zero-set
$$
\tR_S:=\{(z,w)\in\C^2:S(z,w)=0\}\ .
$$
and one has the following standard result
	\begin{lemma}\label{Riemann leaves}
		For any integer $k\ge 1$ and for any polynomial $S\in\Pol$, there exists a set $\cN_S\subset \C$ (defined explicitly in appendix A, see \ref{esclusi}) satisfying $\text{card }\cN_S\le \tN_k$ - where $\tN_k\in\N$ is an upper bound depending only of $k$ - and such that over any simply connected domain $\mathtt D\subset \C$ the intersection of the algebraic curve $\tR_S$ with $\mathtt D\times \C$ is the union of at most $k$ disjoint graphs of holomorphic functions defined over $\mathtt D$ if and only if $\mathtt D\cap\cN_S =\varnothing$.
	\end{lemma}
	The proof of this result can be found by putting together known results on algebraic curves (see e.g. \cite{Narasimhan_1991}). For the sake of clarity, it is given in appendix \ref{provetta}.
	\begin{rmk}
		Following ref. \cite{Nekhoroshev_1973}, the elements of $\cN_S$ are called {\it excluded points}.
	\end{rmk}
	\begin{rmk}
		The number of graphs in Lemma \ref{Riemann leaves} may be equal to zero, for example if $S(z,w)=z$, we have $\tR_S=\{(z,w)\in\C^2: z=0\}$ and the point $z=0$ is excluded by construction (see Appendix \ref{provetta}).
	\end{rmk}
	\begin{defn}[Riemann branches and leaves]
		In the setting of Lemma \ref{Riemann leaves}, if $\mathtt R_S$ is non-empty over $\mathtt D$, the holomorphic functions whose graphs cover $\mathtt D$ are algebraic since  their graphs solve the equation
		$
		S(z,w)=0$ for all $z\in\mathtt D$. These functions will henceforth be called the {\it Riemann branches} of $S$ over $\mathtt D$, whereas their graphs will be referred to as the {\it Riemann leaves} of $S$ over $\mathtt D$.
		
	\end{defn}

  It is a standard fact that, up to constant multiplicative factors, any polynomial $S\in\Pol$ can be uniquely factorized as
 \begin{equation}\label{decomposition_2}
 	S(z,w)=q(z)\,  \displaystyle{\Pi_{i=1}^m} ({\mathcal{S}}_i(z,w))^{j_i}
 \end{equation}
 for some $1\le j_i\le k$, $1\le m\le k$, where the $\mathcal{S}_i$'s are non-constant, irreducible, mutually non-proportional polynomials. Hence, without any loss of generality, we can pass to the unit sphere in $\cQ(k)$ and assume $||q||=1$ for an arbitrary norm $\vert\vert \cdot \vert\vert$.
 
 We denote
 \begin{equation}\label{prime_bis}
\overline{\mathcal S}(z,w):=\displaystyle{\Pi_{i=1}^m}({\mathcal{S}}_i(z,w))^{j_i}
\end{equation}
and we have the polynomial product:
\begin{equation}\label{decompose}
	S(z,w)=q(z)\overline{\mathcal{S}}(z,w).
\end{equation}

We start by giving the following

\begin{defn}\label{defB}
 $\cB=\cB(k,\rho)\subset \Pol$ denotes the set of polynomials $S\in\Pol\backslash\{0\}$ such that the polynomial $\overline \cS$ in decomposition \eqref{decompose} has no excluded points (definition \ref{esclusi})  in $\cD_\rho(0)$.
\end{defn}

\begin{rmk}\label{remarkino}
	Given $S\in\cB$, by decomposition \eqref{decompose} and Definition \ref{esclusi}, the only possible excluded points for $S$ in $\cD_\rho(0)$ are those at which $q(z)=0$. Inside the disk $\cD_\rho(0)$, the algebraic curve $\tR_S$ is therefore the union of at most $k$ elements that can be either disjoint holomorphic Riemann leaves of $\overline \cS$ or vertical lines in $\C^2$ of the kind $z=z_0$, with $q(z_0)=0$. In particular, all the Riemann branches of $S\in\cB$ are holomorphic over $\cD_\rho(0)$.
\end{rmk}

	\begin{rmk}
		The set $\cA$ of Definition \ref{branch} is contained in $\cB$ and, with the notations of Theorem \ref{Bernstein0}, the functions in $\cV(k,\rho)$ are precisely those associated to the polynomials in $\cB$.
\end{rmk}
In order to prove Lemma \ref{chiusura}, we need the following
\begin{lemma}\label{panettone}
	$\cB\cup \{0\}$ is closed in $\Pol$.
\end{lemma}
The proof of Lemma \ref{panettone} is quite technical and requires some intermediate results, which are stated in the sequel.

We start by considering a sequence $\{S_n(z,w)\}_{n\in\N}$ of polynomials in $\cB\cup\{0\}$, converging to a polynomial $S\in\Pol$. We can assume that $S\not\equiv 0$ otherwise there is nothing to prove; hence we have $S_n\not\equiv 0$ for $n$ large enough.

 Following decomposition \eqref{decompose}, we write $S_n(z,w):=q_n(z)\overline \cS_n(z,w)$ and, by construction, the sequence of polynomials $\{q_n\}_{n\in\N}$ is in the compact unit sphere and admits a converging subsequence. In order not to burden notations, with slight abuse, in the sequel we shall indicate this subsequence with the same symbol $\{q_n\}_{n\in\N}$ and we shall denote with $\widehat q$ its limit, which is not identically null by construction.

We recall that $\cN_S$ and $\cN_{S_n}$ (for $n\in\N$) denote the sets of excluded points of $S$ and $S_n$, respectively. For $r>0$ small enough, we remove from $\cD_\rho(0)$ all open neighborhoods of radius $r$ around the excluded points of $S$ and consider the following compact set:
\begin{equation}\label{eerre}
{\mathtt E}_r:=\{z\in \coD_{\rho -r}(0)\ /\ |z-z_0|\geq r\ {\rm for}\ z_0\in \cN_S\}\subset\cD_\rho (0) .
\end{equation}

\begin{lemma}\label{accumulazione}
	There exists $r_0=r_0(\rho, k)$ such that, for any $0<r\le r_0$, one has ${\mathtt E}_r\not=\varnothing$ and there exists an integer $n_0=n_0(r)$ such that:
	\begin{equation}
	{\mathtt E}_r\cap \cN_{S_n}=\varnothing
	\text{  for all } n\geq n_0\ .
	\end{equation}
\end{lemma}
\begin{proof}
		The fact that $\tE_r\neq \varnothing$ for $r$ sufficiently small is an immediate consequence of Definition \ref{eerre} and of the fact that $\text{card}\,\cN_S$ is bounded by a number depending only on $k$ (see Lemma \ref{Riemann leaves}).
	
As for the second part of the statement, since $S_n\longrightarrow S\in\Pol$, and $q_n\rightarrow \widehat q\not\equiv 0$, there exists a polynomial $\widehat S\in\Pol$ such that
	\begin{equation}\label{capraiaelimite}
	\lim_{n\longrightarrow+\infty}S_n(z,w)=
	\lim_{n\longrightarrow+\infty}\overline \cS_n(z,w)\times \lim_{n\longrightarrow+\infty} q_n(z)= \widehat S(z,w)\times \widehat q(z).
	\end{equation}

	By applying again decomposition \eqref{decompose} to $\widehat S$ we obtain $\widehat S(z,w)=\widetilde q(z) \overline \cS(z,w)$, so that we can write $S(z,w)=q(z)\overline \cS(z,w)$ by setting
	\begin{equation}\label{cucu}
	q(z):=\widehat q(z)\times \widetilde q(z) \ .
	\end{equation}
	Therefore, all the roots of $\widehat q$ are also roots of $q$ and belong to $\cN_S$.	
	By construction (see also remark \ref{remarkino}), all  points in $\cN_{S_n}$ are roots of $q_n(z)=0$. Since $q_n\longrightarrow \widehat q$, taking into account the continuous dependence of the roots of a polynomial on its coefficients expressed in Theorem \ref{continuous_dependence}, one has that for sufficiently high $n$ the roots of $q_n$ must be either $r$-close to the roots of $\widehat q $, and hence to some point of $\cN_S$, or outside of the disc of radius $\cD_{1/r}(0)$.
	Taking $r_0<1/\rho$, one has $\cD_{1/r}(0)\supset \cD_\rho(0)$, whence the thesis.
\end{proof}

 We fix $0<r\le r_0$, with $r_0$ the bound in Lemma \ref{accumulazione}, and  we consider a point $z^\star\in \tE_r$, hence $z^\star$ is not an excluded point of $S$ and any solution of $S^{z^\star}(w):=S(z^\star,w)=0$ must belong to the image of a Riemann branch of $S$ holomorphic in a neighbourhood of $z^\star$. We fix one of these branches and denote it with $h$. The continuous dependence of the zeros of a polynomial on its coefficients (Theorem \ref{continuous_dependence}) ensures the existence of a sequence $\{w^\star_n\}_{n\in\N}$ of roots of $ S^{z^\star}_n(w):=S_n(z^\star,w)$ such that
$$
w^\star_n\longrightarrow h(z^\star)\ .
$$
 Lemma \ref{accumulazione} and Remark \ref{remarkino} together with the fact that $S_n\in\cB$ for all $n\in\N$ ensure that, for any fixed $n>n_0 (r)$, the point $(z^\star,w^\star_n)$ must belong to the Riemann leaf of one of the branches of ${\overline \cS}_n$, denoted $h_n$, which is holomorphic over ${\mathcal D}_\rho (0)$. Hence we have the pointwise convergence
\begin{equation}\label{puntuale}
h_n(z^\star)\longrightarrow h(z^\star)\ .
\end{equation}
We show in the sequel that the sequence  $\{h_n\}_{n\in\N}$ admits a subsequence that converges uniformly on any compact subset of $\cD_\rho (0)\backslash\cN_S$ to an holomorphic function which extends $h$ over $\cD_\rho (0)\backslash\cN_S$. In order to prove this claim, which is fundamental to demonstrate Lemma \ref{panettone}, we need the following results.

\begin{lemma}\label{bounded}
	The Riemann branches of $S$ are bounded on the compact sets included in $\cD_\rho (0)\backslash\cN_S$.
\end{lemma}
\begin{proof}
	
	By construction, any point $\widehat z\in\cD_\rho (0)\backslash\cN_S$ is regular for $S$, hence there exists an open neighbourhood $V\subset\C$ of $\widehat z$ such that the algebraic curve $\tR_S\cap \{V\times\C\}$ is composed of at most $k$ graphs of holomorphic functions bounded over $V$. Since any compact set included in $\cD_\rho (0)\backslash\cN_S$ can be covered by a finite number of these neighbourhoods, the claim is proved.
\end{proof}

\begin{lemma}\label{sesto}
	The sequence
	$\{h_n\}_{n\in\N}$ is locally bounded (\ref{bounded families}) over $\cD_\rho (0)\backslash\cN_S$.
\end{lemma}
\begin{proof}
  If, by absurd, there exists a compact $\mathtt K\subset\cD_\rho (0)\backslash\cN_S$ such that $\{h_n\}_{n\in\N}$ is unbounded over $\mathtt K$, then there exists a sequence $\{ z_n\}_{n\in\N}$ in $\mathtt K$ and a strictly increasing function $\varphi$ over $\N$ such that the subsequence $\{\vert h_{\varphi (n)}(z_n)\vert \}_{n\in\N}$ diverges.
	
By Definition \eqref{eerre}, $\cD_\rho (0)\backslash\cN_S=\cup_{r>0}\,\tE_r$, so that there exists $0<r\le r_0$ small enough such that $\mathtt K\subset{\mathtt E}_r\subset\cD_\rho (0)\backslash\cN_S$. Moreover, ${\mathtt E}_r$ is a compact, arc-connected set since it is $\coD_{\rho -r} (0)$ without a finite number of open disks. Then, for any $n\in\N$ it is always possible to construct a continuous arc: 
	$$\gamma_n :[0,1]\longrightarrow{\mathtt E}_r\ {\rm with}\ \gamma_n (0)=z^\star\ {\rm and}\  \gamma_n (1)=z_n .$$ 
	
	We introduce the continuous functions:
	$$
	\psi_{n}:[0,1]\longrightarrow \R\quad ,\qquad \psi_{n}(t):=|h_{\varphi (n)}(\gamma_n(t))|\ .
	$$  
	Since $S_n^z\rightarrow S^z$ uniformly for $z\in\coD_{\rho -r}(0)$, Theorem \ref{continuous_dependence} ensures that, for all $\eps>0$, there exists $\mathtt n(\eps)\in\N$ such that for all $n>\mathtt n (\eps)$ and all $z\in\coD_{\rho -r}(0)$, the roots of $S_n^z$ are either $\eps$-close to the roots of $S^z$ or in the complement of the closed disk $\coD_{1/\eps}(0)$.
	Moreover, Lemma \ref{bounded} ensures that the roots of $S^z$ are uniformly bounded for all $z\in{\mathtt E}_r$.   We indicate with $w_{max}(r)$ the maximal module that the Riemann branches of $S$ can reach on ${\mathtt E}_r$ and we set
	$$\eps_0 (r) =\frac{1}{w_{max}(r)+1}.$$
	
	In this setting, we can consider a fixed integer $n>{\mathtt{n}(\eps_0 (r))}$ such that:
	\begin{equation}\label{siparte}
	\psi_{n}(1)>\frac{1}{\eps_0}=w_{max}(r)+1
	\end{equation}
	and taking \eqref{puntuale} into account, we also assume that $n$ is high enough to ensure:
	\begin{equation}\label{siresta}
	|\,\psi_{n}(0)-|h(z^\star)|\,|=|\,|h_{\varphi(n)}(z^\star)|-|h(z^\star)|\,|\le|h_{\varphi(n)}(z^\star)-h(z^\star)|<\eps_0 <1
	\end{equation}
	hence $\psi_{n}(0)<|h(z^\star)| +1\leq w_{max}(r)+1$.
	
	With \eqref{siparte} and \eqref{siresta}, the intermediate value theorem applied to $\psi_{n}$ implies that there exists $t_\star\in ]0,1[$ satisfying
	\begin{equation}\label{zumpapa}
	\psi_{n}(t_\star)=w_{max}(r)+1
	\ .
	\end{equation}
	
	But $n>\mathtt n (\eps_0(r))$ and $\gamma_{n}(t_\star)\in \coD_{\rho -r}(0)$, hence $h_{\varphi (n)}(\gamma_n(t_\star))$ is either in the complement of the closed disk $\coD_{1/\eps_0}(0)$ and
	\begin{equation}\label{alligalli1}
	\psi_{n}(t_\star)>\frac{1}{\eps_0}=w_{max}(r)+1
	\end{equation}
	or $\eps_0$-close to the roots of $S^z$ and
	\begin{equation}\label{alligalli2}
	\psi_{n}(t_\star )<w_{max}(r) +\eps_0 < w_{max}(r)+1.
	\end{equation}
	
	Conditions \eqref{zumpapa}, \eqref{alligalli1} and \eqref{alligalli2} are in contradiction and the statement is proved.
\end{proof}

\begin{lemma}\label{convergenza}
	
	There exists a subsequence of $\{h_n\}_{n\in\N}$ which converges uniformly on the compact subsets of $\cD_\rho (0)\backslash\cN_S$ to a Riemann branch of $S$ (still denoted $h$) extending holomorphically $h$ over $\cD_\rho (0)\backslash\cN_S$.
\end{lemma}
\begin{proof}
	With Lemma \ref{sesto} and Montel's Theorem \ref{montel}, it is possible to extract a subsequence - still denoted $\{h_n\}_{n\in\N}$ with slight abuse of notation - that converges uniformly on the compact subsets of $\cD_\rho (0)\backslash\cN_S$ to a function holomorphic over $\cD_\rho (0)\backslash\cN_S$ which is also still denoted $h$.   Finally, thanks to Lemma \ref{sequence} and to the fact that $S_n\longrightarrow S$, one has $S(z,h(z))=0$ for any $z\in \cD_\rho(0)$.\end{proof}

\begin{rmk}
	By the above Lemma, $\cN_S$ does not contain any ramification points.
\end{rmk}

With the help of Lemma \ref{convergenza}, we are now able to prove Lemma \ref{panettone}.

\begin{proof}{\it(Lemma \ref{panettone})}
	
	The aim is to prove that the set of excluded points $\cN_{\overline \cS}$ for ${\overline \cS}$ associated to the limit polynomial $S$ is empty, from which the thesis follows.
	
	Assume that $z_0\in\cN_{\overline \cS}$. Since $\cN_{\overline \cS}$ is a finite set, for $t>0$ small enough the punctured disc $\disp{t}{z_0}:=\{z\in\cD_\rho(0): 0<|z-z_0|<t\}$ is included in $\cD_\rho (0)\backslash\cN_{\overline \cS}$ and any branch $h$ of the polynomial ${\overline \cS}$ is holomorphic in $\disp{t}{z_0}$. Then, by Laurent's Theorem \ref{classificazione} and Proposition \ref{essential}, $z_0$ is either a removable singularity or a pole. We show that the second possibility does not occur.
	
	If by absurd $z_0$ is a pole for $h$, then $\lim_{z\longrightarrow z_0}h(z)$ is infinite and one can choose the radius $t$ small enough so that $h(z)\neq 0$ for all $z\in\disp{t}{z_0}$. Hence the function $\phi :=1/h$ is analytic on the punctured disc $\disp{t}{z_0}$ and it is also bounded since its limit is zero when $z$ goes to $z_0$. By Riemann's Theorem (\ref{rrs}) on removable singularities, $\phi$ admits a holomorphic extension, still denoted  $\phi$, on the whole disc $\dis{t}{z_0}$ fulfilling $\phi (z_0)=0$.
	
	Lemma \ref{convergenza} ensures that there exists a subsequence $\{h_{n_j}\}_{j\in\N}$ of Riemann branches for $S_{n_j}$ (actually, by remark \ref{remarkino}, the branches $h_{n_j}$ are analytic over $\cD_\rho (0)$ since $S_{n_j}\in\cB$) which converges uniformly to $h$ on the compact subsets of  $\disp{t}{z_0}\subset\cD_\rho (0)\backslash\cN_S$. Consequently, the functions $h_{n_j}$ do not vanish on any compact subset of the disk
	$\cD_t (z_0)$ for $j$ large enough. This ensures that the functions
	$\{\phi _{n_j}\}_{j\in\N}:=\{1/h_{n_j}\}_{j\in\N}$ are holomorphic on $\cD_t (z_0)$. Moreover, the sequence $\{\phi_{n_j}\}_{j\in\N}$ converges locally uniformly to $\phi$ on $\disp{t}{z_0}$. Since both $\phi_{n_j}$ and $\phi$ are holomorphic at $z_0$, by the Maximum Principle, this convergence is actually locally uniform over the whole disc $\dis{t}{z_0}$.
	
	On the one hand, we have $\phi(z_0)=0$ and $\phi(z)\neq 0$ for $z\in \disp{t}{z_0}$, since in this domain $\phi(z)=1/h(z)$ and $h$ is holomorphic on $\disp{t}{z_0}$.
	
	On the other hand, the terms of the subsequence $\{\phi_{n_j}\}_{j\in\N}:=\{1/h_{n_j}\}_{j\in\N}$ are nowhere-vanishing on $\cD_t(z_0)$ and $\phi_{n_j}$ is holomorphic in that domain. Consequently, by Hurwitz's Theorem on sequences of holomorphic functions (\ref{hurwitz}), $\phi$ must be either identically zero or nowhere null on $\dis{t}{z_0}$.
	
	We have obtained a contradiction and therefore $\lim_{z\longrightarrow z_0}h(z)$ is finite.
	By applying once again Riemann's Theorem on removable singularities, $h$ admits an analytic extension $\widetilde h$ to the whole disc $\dis{t}{z_0}$. Moreover, $\widetilde h$ is a Riemann branch of $\overline \cS$ in the whole disc $\dis{t}{z_0}$, since
	$$
	\overline \cS(z_0,\widetilde h(z_0))=\lim_{z\longrightarrow z_0} \overline \cS(z,  h(z))=0\ .
	$$
	
	It remains to rule out the possibility that $z_0$ is singular because the graphs of two distinct branches $h$ and $\ell$ of the limit polynomial ${\overline \cS}$ intersect on it. Assume that 
	$h(z_0)=\ell(z_0)$. By Lemma \ref{convergenza} and by the previous arguments, there exist two subsequences $\{ h_{n_j}\}_{j\in\N}$ and $\{ \ell_{n_j}\}_{j\in\N}$ of branches associated to $\{ {\overline \cS}_{n_j}\}_{j\in\N}$ that approach respectively $h$ and $\ell$ locally uniformly over $\dis{t}{z_0}$.
	
	We first notice that $h_{n_j}$ is distinct from  $\ell_{n_j}$ for
	$j$ large enough, otherwise there exists a subsequence of common branches $h_{n_j}\! =\! \ell_{n_j}$ for ${\overline \cS}_{n_j}$ up to infinity which converges locally uniformly in $\cD_{t}(z_0)$ respectively to $h$ and $\ell$. Consequently $h =\ell$ over $\cD_{t}(z_0)$, which contradicts the assumption that $h$ and $\ell$ are distinct.  Moreover, since $\tR_{{\overline \cS}_{n_j}}$ is composed of distinct regular leaves over $\cD_{t}(z_0)$  for any $j\in\N$, the functions $h_{n_j}\! - \ell_{n_j}$ never vanish over $\cD_t(z_0)$.
	
	Consequently, Hurwitz theorem \ref{hurwitz} ensures that the sequence of holomorphic functions $\{ h_{n_j}\! - \ell_{n_j}\}_{j\in\N}$ converges to a limit which either never vanishes or is identically zero over $\cD_{t}(z_0)$. Here  $\lim_{j\rightarrow +\infty}(h_{n_j}\! - \ell_{n_j})(z_0)=(h-\ell)(z_0)=0$, so $h = \ell$ everywhere over $\cD_{t}(z_0)$, which is again in contradiction with the assumption that $h$ and $\ell$ are distinct.
	
	Therefore, we have proved that the algebraic curve of the limit polynomial ${\overline \cS}$ is composed of disjoint and regular Riemann leaves over a neighbourhood of $z_0$. Since the above arguments hold for any $z_0\in\cN_{S}$, the algebraic curve $\tR_{\overline \cS}$ is composed of distinct leaves over $\cD_\rho(0)$ and the branches of ${\overline \cS}$ are globally holomorphic over $\cD_\rho (0)$, consequently $\cN_{\overline \cS}=\varnothing$ (see Remark \ref{remarkino}).  
\end{proof}

Lemma \ref{panettone} is the cornerstone to prove Lemma \ref{chiusura}.

\begin{proof}{\it (Lemma \ref{chiusura})}
	
 We start by proving the closure of $\cA\cup\{0\}$ in $\Pol$ and consider a sequence $\{S_n\}_{n\in\N}$ in $\cA\cup\{0\}$ which converges to a limit $S\in\Pol$. One has $S\in\cB\cup\{0\}$, since $\cA\cup\{0\}\subset\cB\cup\{0\}$ and $\cB\cup\{0\}$ is closed by Lemma \ref{panettone}. 
	
	By hypothesis, for any fixed $n\in\N$ there exists a Riemann branch $g_n(z)$ which is analytic on $\dis{\rho}{0}$ and satisfies
	\begin{equation}\label{proprieta}
	S_n(z,g_n(z))=0\ ,\ \ g_n(0)=0\ ,\ \ \underset{z\in\cK}{\max}| g_n(z)|=1\ .
	\end{equation}
	
	If $S\equiv0$ there is nothing to prove. 
	
	If $S\not\equiv 0$, we claim that  $\{ g_n\}_{n\in\N}$ has a subsequence that converges uniformly on the compact subsets of $\cD_\rho(0)$ to a branch $g_S$ of $S$ having the good properties.
	
	Infact, since $S\in\cB\cup\{0\}$, the elements of the set $\cN_S$ are the roots of $q(z)=0$. Consequently, $\text{card }\cN_S\le k$ and $\text{card}\,\cK>k$ ensures that there exists $z^\star\in\cK\backslash\cN_S$ such that $\{ g_n(z^\star) \}_{n\in\N}$ is bounded. Up to the extraction of a subsequence, $g_n(z^\star)$ converges to a complex value $w^\star$. Moreover, since $z^\star$ is not an excluded point of $S\in\cB$, we can ensure that $(z^\star,w^\star)$ belongs to a Riemann leaf of $\tR_{\overline \cS}$ which is associated to a holomorphic Riemann branch over $\cD_\rho (0)$, denoted $g_S$.
	
	By the same arguments used in the proofs of Lemmas \ref{sesto} and \ref{convergenza}, the sequence $\{ g_n\}_{n\in\N}$ admits a subsequence which converges uniformly on the compact subsets of $\cD_\rho (0)\backslash\cN_S$ to a Riemann branch $f_S$ associated to $S$. The holomorphy of $f_S$ over $\cD_\rho (0)$ (since $S\in\cB$) and the Maximum Principle imply that the convergence is actually locally uniform on the whole set $\cD_\rho(0)$. Then, by the  unicity of the limit, we have $g_S(z^\star)=f_S(z^\star)$, which implies $g_S\equiv f_S$ over $\cD_\rho(0)$ because $S\in \cB$. This yields $\max_\cK| g_S |=1$ and $g_S(0)=0$, hence $g_S$ meets the requirements of Definition \ref{branch} and $S\in\cA$.
	
Finally, it remains to prove that the function $\mu_\Omega$ in Lemma \ref{chiusura} is continuous. Since  $\{g_{n}\}_{n\in\N}$ converges locally uniformly to $g_S$ in $\cD_\rho(0)$, we can write
	\begin{equation}
	\label{triangolo}
	\lim_{n\longrightarrow +\infty}\left|\,\max_{z\in \cK'}| g_S(z)|-\max_{z\in \cK'}| g_{n}(z)|\,\right|\le \lim_{n\longrightarrow +\infty}\left(\max_{z\in \cK'}| g_S(z)-g_{n}(z)|\right)=0\ ,
	\end{equation}
	for any compact $\cK'\subset \cD_\rho(0)$. By taking $\cK'\equiv \overline\Omega\subset \cD_\rho(0)$, we have
	$$
	\mu_{\Omega}(S):=\max_{z\in \overline\Omega}| g_S(z)|=\lim_{n\longrightarrow+\infty}\Bigg(\max_{z\in \overline\Omega}\left| g_{n}(z)\right|\Bigg)=:\lim_{n\longrightarrow+\infty}\Bigg(\mu_{\Omega}(S_{n})\Bigg)\ ,
	$$
	which implies that $\mu_\Omega$ is continuous. This concludes the proof of Lemma \ref{chiusura}.
	
\end{proof}

\appendix
\section{Proof of Lemma \ref{Riemann leaves}}\label{provetta}
We start by stating two standard results of algebraic geometry.
\begin{lemma}\label{tiodio1}
	For any couple of positive integers $k_1,k_2$ consider two non-zero irreducible, non proportional polynomials $Q_1\in \mathcal P(k_1)$ and $Q_2\in \mathcal P(k_2)$. Then the system $Q_1(z,w)=Q_2(z,w)=0$ has at most $k_1\times k_2$ solutions.
\end{lemma}

\begin{lemma}\label{tiodio2}
	For $k\ge2$, let $Q(z,w)\in \Pol$ be an irreducible polynomial. Then
	$$
	\text{card}\ \{z\in\C\ | \ \exists w\in\C: Q(z,w)=\d_w Q(z,w)=0\}\le k\ .
	$$
\end{lemma}
The first Lemma is a simple corollary of Bézout Theorem (see e.g. \cite{Kendig_2012}, Th. 3.4a), while the second Lemma is also known (see e.g. Proposition 1 and its proof in \cite{Narasimhan_1991}).

With these tools, we can now give the proof of Lemma \ref{Riemann leaves}.
\begin{proof}
	The lemma is trivial if $S$ depends only on $w$ since we have $\cN_S=\varnothing$ in this case because $\tR_S$ is composed of a finite number of Riemann branches which are horizontal lines over the $z$-axis.
	
	If $S\in\Pol$ depends only on $z$, then $\tR_S=\{(z,w)\in\C^2: z=z_0, \text{ with } S(z_0)=0\}$ and the thesis holds true since there are  only vertical lines at the distinct roots of $S$ (whose number is bounded by $k$) and no Riemann branches.
	
	Let's now examine the case in which $S$ depends on both variables where, up to multiplication by constant factors, any polynomial $S\in\Pol$ can be factorized uniquely as
	\begin{equation}\label{decomposition}
		S(z,w)=q(z)\,  \displaystyle{\Pi_{i=1}^m} ( S_i(z,w))^{j_i}
	\end{equation}
	for some $1\le j_i\le k$, $1\le m\le k$ and the ${S}_i$ are non-constant, irreducible, mutually non-proportional polynomials.
	
	We denote
	\begin{equation}\label{prime}
		\overline \cS(z,w)=\displaystyle{\Pi_{i=1}^m} ( S_i(z,w))^{j_i}\ {\rm and}\ \widetilde \cS(z,w)=\displaystyle{\Pi_{i=1}^m} S_i(z,w)
	\end{equation}
	and $\overline \cS^z(w):=\overline \cS(z,w)$, $\widetilde \cS^z(w):=\widetilde \cS(z,w)$ hence $\widetilde \cS^z\in \C [z][w]$ - with $\deg (\widetilde \cS^z)=\ell$ $\ell\in\{1,\ldots ,k\}$ - and $a_\ell (z)$ is the corresponding leading coefficient.
	
	We notice that decomposition \ref{decomposition} and definition \ref{prime} ensure that $\tR_S$ is the union of the vertical lines $z=z^*$ with $q(z^*)=0$ and of the Riemann surface $\tR_{\overline \cS}$, moreover the Riemann surfaces $\tR_{\overline \cS}$ and $\tR_{\widetilde \cS}$ are identical.
	
	\begin{defn}[Excluded points]\label{esclusi}
		Taking decomposition \ref{decomposition} into account, we define $\cN_S\subset\C$ as the set of those points $z_0\in\C$ that satisfy at least one of the following conditions
		\begin{enumerate}
			\item $$
			q(z_0)=0\qquad \text{(Vertical lines)}
			$$
			\item There exists $w_0\in\C$ such that for some $i\in\{1,...,m\}$
			$$
			\begin{cases}
				S_i(z_0,w_0)=0\\
				\d_w S_i(z_0,w_0)=0\\
			\end{cases}
			\qquad \text{(Ramification points)}
			$$
			\item There exists $w_0\in\C$ such that for some $i,j\in\{1,...,m\}, i\neq j$
			$$
			\begin{cases}
				 S_i(z_0,w_0)=0\\
				  S_j(z_0,w_0)=0\\
			\end{cases}
			\qquad \text{(Intersection of graphs)}
			$$
			\item $$
			a_\ell (z_0)=0\qquad \text{(Poles)}
			$$
		\end{enumerate}
	\end{defn}
	Henceforth, we prove that over $\C\backslash \cN_S$ the conclusions of Lemma \ref{Riemann leaves} are valid and that we can choose the set $\cN_S$ as the excluded points for $S$.
	
	To see this, we fix a point $z^*\in \C\backslash \cN_S$.
	
	By negation of condition (1), we have $q(z)\neq 0$ in the vicinity of $z^*$, hence the vertical lines are excluded from the algebraic curve $\tR_S$ in the vicinity of $z^*$.
	
	Then, we also notice that for any value $w^*$ such that $\overline \cS^{z^*}(w^*)=0$, by decomposition \eqref{decomposition} and negation of condition (3), one must have $ S_i(z^*,w^*)=0$, for exactly one $i\in\{1,...,m\}$. Hence, by negation of condition (2) at $(z^*,w^*)$, we can apply the implicit function theorem and there exists an open neighbourhood $V$ around $(z^*,w^*)$ such that $\tR_S\cap V=\tR_{\overline \cS}\cap V=\tR_{\widetilde \cS}\cap V=\tR_{S_i}\cap V$ is the graph of an unique holomorphic function.
	
	Finally, the negation of condition (4) in a neighbourhood of $z^*$ ensures that the polynomial $\widetilde \cS^{z}$ admits $\ell\leq k$ complex roots counted with multiplicity for $z$ in the vicinity of $z^*$. With $z^*\in \C\backslash \cN_S$, a direct computation ensures that the discriminant of $\widetilde \cS^{z^*}$ is non zero since $\widetilde \cS^{z^*}$ and its derivative cannot have common roots, hence $\widetilde \cS^{z}$ admits simple roots for $z$ in the vicinity of $z^*$.
	
	This implies the existence of a neighbourhood $V$ of $z^*\in \C\backslash \cN_S$ such that the algebraic curve $\tR_S\cap V\times\C=\tR_{\widetilde \cS}\cap V\times\C$ is the union of exactly $\ell\leq k$ disjoint graphs of holomorphic branches.
	
	Hence, over a simply connected domain $\mathtt D\subset\C\backslash \cN_S$, branch cuts can be avoided and the Riemann surface $\mathtt R_S$ is the finite union of at most $k$ disjoint graphs of holomorphic functions.
	
	\medskip
	
	Conversely, consider a simply connected complex domain  $\mathtt D$ such that $\mathtt R_S\cap \mathtt D\times\C$ is the finite union of $\ell\in\{ 1,\ldots ,k\}$ disjoint graphs of functions $h_1(z),\ldots ,h_\ell (z)$ holomorphic over $\mathtt D$.
	
	For a fixed point $z^*\in\mathtt D$, the polynomial $S^{z^*}$ admits $\ell$ roots, and decomposition \ref{decomposition} ensures that we have $q(z^*)\not= 0$. Moreover,  the discriminant of ${\widetilde S}^z$ might be zero only at a finite number of points (since the discriminant is itself a polynomial) but we always have $\ell$ roots for ${\widetilde S}^z$ with $z\in\mathtt{D}$ as a consequence of our assumption that the Riemann leaves are distinct. Hence, the roots are simple and the degree of ${\widetilde S}^z$ is constant equal to $\ell$ for all $z\in\mathtt D$. Consequently, the discriminant of $\widetilde \cS^z$ is non-zero and $a_\ell (z)\not= 0$ for all $z\in\mathtt D$.
	
	Moreover, for any $i,j\in\{1,...,m\}, i\neq j$ and for any $w\in\C$, we have either $S_i(z^*,w)\not= 0$ or $S_j(z^*,w)\not= 0$ otherwise two distincts Riemann leaves associated respectively to $S_i$ and $S_j$ would intersect. 
	
	Finally, we have the decomposition $\widetilde \cS(z,w)=a_\ell(z)(w-h_1(z))\ldots (w-h_\ell (z))$
	for $(z,w)\in\mathtt D\times\C$, and for $z\in\mathtt D$ we can check that $S^z$ and its derivative cannot have common roots under our assumptions. Hence, for any $w\in\C$ and any $i\in\{1,...,m\}$, we have either $S_i(z^*,w)\not= 0$ or $\d_w S_i(z^*,w)\not= 0$.
	
	\medskip
	
	Then, we prove that the cardinality of $\cN_S$ is bounded by a quantity depending only on $k$. 
	
	 Conditions (1) and (4) are polynomial equations of degree less or equal to $k$ ($q$ factorizes all the terms in $z$ and $a_\ell (z)$ is the coefficient of the term of highest degree in $w$), hence they have at most $k$ solutions. By Lemma \ref{tiodio1}, condition (2) is satisfied at most at $k$ points. Since the index $i$ in (2) can assume at most $k$ values, this condition yields $k^2$ singularities. In the same way, Lemma \ref{tiodio2} says that condition (3) is satisfied at most at $k^2$ points. Since the indices $i,j$ in condition (3) can take each at most $k$ values and $i\neq j$, we get $k^2\binom{k}{2}$ solutions. The sum of the previous estimates yields a bound depending only on $k$.
	\end{proof}

\section{Tools of complex analysis}
$\cO(\cD)$ denotes the set of holomorphic functions over an open domain $\cD\subset\C$.
\begin{defn}[Normal families]
	A family $\cF\subset\cO(\cD)$ is said to be normal iff it is precompact, that is if from any sequence in $\cF$, one can extract a subsequence converging uniformly on the compact subsets of $\cD$ to a function in $\cO(\cD)$.
\end{defn}
\begin{defn}[Locally bounded families]\label{bounded families}
	A family $\cF\subset\cO(\cD)$ is said to be locally bounded if for any compact set $\cK\subset \cD$, there exists a constant $C_\cK\ge 0$ such that $\max_\cK|f|\le C_\cK$ for all $f\in\cF$.
\end{defn}
A criterion for establishing if a given family $\cF$ is normal is Montel theorem, which is a version of Ascoli-Arzela theorem for sets of holomorphic functions
\begin{thm}[Montel, see for example \cite{Taylor_2019}, Prop. 21.3]\phantom{1}
	
	A family $\cF\subset \cO(\cD)$ is locally bounded if and only if it is normal.
\end{thm}

From this theorem comes an important
\begin{cor}\label{montel}
	If $\{f_n\}_{n\in\N}$ is a locally bounded sequence of holomorphic functions over $\cD$, then one can extract a subsequence that converges uniformly to a holomorphic function $f$ on all the compact subsets of $\cD$.
\end{cor}

Hereafter, we state some further classic results of complex analysis.

\begin{thm}[Laurent series, \cite{Taylor_2019} p. 130]\label{Laurent}\phantom{1}
	
	Take two numbers $0<r_1<r_2$, $z_0\in\C$ and $f\in \cO(\tA_{r_1,r_2}(z_0))$ where $\tA_{r_1,r_2}$ is the annulus of radii $r_1$ and $r_2$ around $z_0$, then
	$$
	f(z)=\sum_{k=-\infty}^{+\infty}a_k (z-z_0)^k\ ,
	$$
	where the convergence is uniform over the compact subsets of $\tA_{r_1,r_2}(z_0)$.
	
\end{thm}
\begin{cor}[Laurent classification of singularities, \cite{Taylor_2019}]\label{classificazione}\phantom{1}
	
	Let $z_0\in\C$ be a singularity of a holomorphic function $f$ and consider the Laurent series $f(z)=\sum_{k=-\infty}^{+\infty}a_k (z-z_0)^k$, then
	\begin{itemize}
		\item $z=z_0$ is a removable singularity iff $a_k=0$ for any $k\le -1$.
		\item given $m\in\N$, $z=z_0$ is a pole of order $m$ iff $a_{-m}\neq 0$ and $a_k=0$ for any $k\le -m-1$.
		\item $z=z_0$ is an essential singularity iff $a_k\neq 0$ for infinitely many negative integers $k$.
	\end{itemize}
\end{cor}

\begin{thm}[Casorati-Weierstrass Theorem, \cite{Taylor_2019} p.127]\label{casorati}\phantom{1}
	
	Let $z_0\in\C$ be an isolated essential
	singularity of a function $f\in \cO (\Omega \backslash\{ z_0\})$, then for every neighborhood $V$ of $z_0$ in $\Omega$, the image of $V \backslash\{ z_0\}$ under $f$ is dense in $\C$.
\end{thm}
\begin{thm}[Riemann's on removable singularities, \cite{Taylor_2019} p.126]\label{rrs}\phantom{1}
	
	Take $r>0$, $z_0\in\C$ and $f\in \cO(\disp{r}{z_0})$ with $f$ bounded.
	
	Then
	
	\begin{itemize}
		\item $\lim_{z\longrightarrow z_0}f(z)$ exists and is finite;
		\item the function $\tilde f:\dis{r}{z_0}\longrightarrow \C$ ,
		$$
		\tilde{f}:=
		\begin{cases}
		f(z)\qquad &\text{  if } z\in\disp{r}{z_0}\\
		\lim_{z\longrightarrow z_0}f(z)\qquad &\text{  if } z=z_0\\
		\end{cases}
		$$
		is holomorphic.
	\end{itemize}
	
\end{thm}
\begin{thm}[Hurwitz, \cite{Taylor_2019} p. 216]\label{hurwitz}
	Suppose that $\cD$ is a connected open
	set and that $\{f_n\}_{n\in\N}$ is a sequence of nowhere vanishing holomorphic functions over $\cD$. If the sequence $\{f_n\}_{n\in\N}$ converges uniformly
	on compact subsets of $\cD$ to a (necessarily holomorphic) limit
	function $f$, then either $f$ is nowhere-vanishing over $\cD$ or $f$ is identically $0$.
\end{thm}
\section{Tools of algebraic geometry and applications}
\subsection{On the dependence of the roots of a polynomial on its coefficients}

It is a standard fact in the study of algebra that the roots of a monic complex polynomial of one variable depend continuously on its coefficients. The same result holds true for non-monic polynomials once one takes solutions at infinity into account by the means of the compact identification of $\C\cup\{\infty\}$ with the Riemann sphere. Without entering into too many details, we state the following result, whose proof can be found in \cite{Cucker_Gonzalez-Corbalan_1989}.
\begin{thm}\label{continuous_dependence}
	Let $P(w)=a_n w^n+a_{n-1}w^{n-1}+...+a_0$ be a non-zero complex polynomial of degree $k\le n$. Let $\xi_1,...,\xi_r$ be its roots in $\C$ with $m_1,...,m_r$ their respective multiplicities. Fix $\eps>0$ small enough and denote with  $\cD_\eps(\xi_1),...,\cD_\eps(\xi_r)$ the disjoint disks of radius $\eps$ centered at $\xi_1,...,\xi_r$, such that $\cD_\eps(\xi_j)\subset\cD_{1/\eps}(0)$ for all $j\in\{1,...,r\}$. Then, there exists $\delta(\eps)>0$ such that every complex polynomial $Q(w)=b_n w^n+b_{n-1} w^{n-1}+...+b_0$ satisfying $|b_j-a_j|<\delta(\eps)$ for all $j\in\{ 0,...,n\}$ has $m_i$ roots (counted with multiplicity) in each $\cD_\eps(\xi_i)$ for $i\in\{1,...,r\}$ and $\deg(Q)-k$ roots belonging to the complementary of $\cD_{1/\eps}(0)$.
\end{thm}
This theorem formalizes the intuitive idea that, if one takes a polynomial 
$$Q(w)=a_n w^n+a_{n-1}, w^{n-1}+...+a_0,\ {\rm with}\ a_n\neq 0$$
 and makes $a_{k+1},a_{k+2},...,a_n$ tend to zero while $a_k\neq 0$, then $n-k$ solutions "continuously go to infinity" and $k$ solutions, counted with their multiplicities, "stay finite".

\subsection{Application to sequences of algebraic functions}
	
\begin{lemma}\label{sequence}
	Take an open bounded set $\mathtt U\subset \C$, let $\{g_n\}_{ n\in{\mathbb N}}$ be a sequence of algebraic functions in $\cO (\mathtt U )$ associated to polynomials of degree $k\in \N$ and converging in $\mathtt U$ to a function $g\in \cO(\mathtt U)$. Then $g$ is an algebraic function. Moreover, there exists a sequence of polynomials $\{Q_n\in \Pol\}_{n\in\N}$ solving the graphs of the functions in $\{g_n\}_{ n\in{\mathbb N}}$ which converges to a polynomial $Q\in \C[z,w]$ that solves $\text{graph}(g)$ everywhere in $\mathtt U$.  
	
\end{lemma}

\begin{proof}

For any $n\in\N$, the graph of the function $g_{n}$ satisfies $S_{n}(z,g_{n}(z))=0$ for some polynomial $S_{n}\in\cP(k)\backslash\{ 0\}$ and the equation $S_{n}(z,g_{n}(z))=0$ is invariant when $S_{n}$ is multiplied by any non-zero constant. Without loss of generality, one can choose an arbitrary norm $||\cdot||$ in $\Pol\simeq \C^m$, with $m=(k+1)(k+2)/2$, and consider the sequence of polynomials $\{Q_n\}_{ n\in{\mathbb N}}$ associated to $\{g_n\}_{ n\in{\mathbb N}}$ by defining, for any $n\in\N$:
\begin{equation}\label{Q}
Q_{n}(z,w):=\frac{S_{n}(z,w)}{||S_{n}||}\ {\rm \ with}\ Q_{n}(z, g_{n}(z))=0\ {\rm \ and}\ Q_{n}\in\mathbb S^m
\end{equation}
 where $\mathbb S^m$ denotes the unitary sphere in $\Pol\simeq \C^m$. By compacity of ${\mathbb S}^m$, there exists a subsequence $\{Q_{n_{j}}\}_{j\in\N}$ converging to a polynomial $Q\in\mathbb S^m$. Moreover, if we denote $Q_{n_{j}}^z(w):=Q_{n_{j}}(z,w)$ and $Q^z(w):=Q(z,w)$ - hence $Q_{n_{j}}^z$ and $Q^z$ belong to $\cQ(k)$ for any fixed $z\in \C$ - we have the convergence
\begin{equation}\label{uniforme_polinomi}
	\lim_{j\longrightarrow+\infty }||Q^{z^*}_{n_{j}}-Q^{z^*}\|=0
\end{equation}
for any fixed $z^*\in{\mathtt U}$ and with respect to any norm in $\cQ(k)$.

 The convergence in \eqref{uniforme_polinomi} and Theorem \ref{continuous_dependence} imply that the sequence $\{g_{n_{j}}(z^*)\}_{j\in \N}$ approaches a root of $Q^{z^*}$ for any $z^*\in{\mathtt U}$. Since $g_{n_{j}}$ converges over $\mathtt U$ to $g$, then $g(z)$ is a solution of $Q^z(w)=0$ for any $z\in\mathtt U$.
  
Finally, since $g\in\cO(\mathtt U)$, then $g$ is a Riemann branch of $Q\in\cP (k)$ over $\mathtt U$. 
\end{proof}
\subsection{Non-existence of essential singularities for algebraic functions}

\begin{prop}\label{essential}
	An algebraic function $f$ cannot have any essential singularity.
\end{prop}
\begin{proof}
	By Lemma \ref{Riemann leaves} and decomposition \ref{decomposition}, the singularities of $f$ are included in the set $\cN_{\overline \cS}$ of excluded points (see \ref{esclusi}). Hence, suppose by absurd that $z_0\in\cN_{\overline \cS}$ is an essential singularity. Since the cardinality of $\cN_{\overline \cS}$ is finite, $z_0$ is isolated. Then, Casorati-Weierstrass Theorem (see Th. \ref{casorati}) holds and, for any fixed $w_0\in\C$, one can build a sequence $\{z_k\}_{k\in\N}$ converging to $z_0$ and such that
	$$
	\lim_{k\longrightarrow+\infty}f(z_k)=w_0\ .
	$$
	But $w_0$ is also a root of the one variable polynomial $\overline \cS^{z_0}(w):=\overline \cS(z_0,w)$ since $f(z)$ is a Riemann branch of $\overline \cS$ in a punctured neighborhood centered at $z_0$ and
	$$
	\overline \cS^{z_0}(w_0):=\overline \cS(z_0,w_0)=\lim_{k\longrightarrow+\infty} \overline \cS(z_k,f(z_k))=0.
	$$
	
	This construction holds for any $w_0\in\C$ and the polynomial $\overline \cS^{z_0}$ is null but, necessarily, $(z-z_0)$ is a factor of $\overline{\cS}$ and this is not possible with decomposition \ref{decomposition}. 
	Hence, $f(z)$ cannot have an essential singularity at $z_0$.	
\end{proof}

\medskip

\section*{Acknowledgements}
The authors are extremely grateful to L. Biasco and J.P. Marco for the useful conversations on the subject as well as for the suggestions about the bibliography.

\bibliographystyle{plain}
\bibliography{SteepNied}
\end{document}